\documentclass
[
    a4paper,
    DIV=11,
    abstracton
]
{scrartcl}
\usepackage{fullpage,authblk,amsmath,amssymb,amsthm,amsfonts}
\usepackage{bm, bbm, dsfont}
\usepackage{enumerate}
\usepackage{tikz,algorithm2e}
\usepackage{caption}
\usepackage{subcaption}
\usepackage{url,mathtools}
\usepackage{footmisc}

\usepackage[pdffitwindow=true,
		   pdfstartview={FitH},
            plainpages=false,
            pdfpagelabels=true,
            pdfpagemode=UseOutlines,
            pdfpagelayout=SinglePage,
            bookmarks=false,
            colorlinks=true,
            hyperfootnotes=false,
            linkcolor=blue,
            urlcolor=blue!30!black,
            citecolor=green!50!black]{hyperref}
            
\usepackage{color}
\usepackage[colorinlistoftodos,prependcaption,textsize=tiny]{todonotes}

\theoremstyle{plain}
\newtheorem{thm}{Theorem}[section]
\newtheorem{lem}[thm]{Lemma}
\newtheorem{prop}[thm]{Proposition}

\theoremstyle{definition}

\numberwithin{equation}{section}

\newcommand{\inter}[1]{\textup{int}({#1})}

\newcommand{\diam}[1]{\textup{diam}({#1})}
\newcommand{\inrad}[1]{\textup{inrad}({#1})}

\newcommand{\abs}[1]{\left\vert#1\right\vert}

\newcommand{\bb}[1]{\mathbb{#1}}

\begin{document}
\title{On a degeneracy ratio for bounded convex polytopes}

\author{Nada Cvetkovi\'{c}\thanks{\ Corresponding author: ncvetkovic@math.fu-berlin.de}}
\affil{Department of Mathematics and Computer Science, Freie Universit\"at Berlin, Arnimallee 6, 14195 Berlin, Germany}

\author{Han Cheng Lie\thanks{\ hanlie@uni-potsdam.de}}
\affil{Universit\"{a}t Potsdam, Institut f\"{u}r Mathematik, Karl-Liebknecht Str. 24-26, D-14476 Potsdam OT Golm, Germany}
\renewcommand\Affilfont{\small}

\date{}

\maketitle

\begin{abstract}
We consider a quantity that measures the roundness of a bounded, convex $d$-polytope in $\bb{R}^d$. We majorise this quantity in terms of the smallest singular value of the matrix of outer unit normals to the facets of the polytope.
\end{abstract}



\section{Introduction}

A well-known problem in convex optimisation involves finding the radius of the largest inscribed ball of a bounded convex $d$-polytope $P\subset\mathbb{R}^d$. For simplicity, we shall refer to bounded convex $d$-polytopes as `polytopes'. We shall refer to the above-mentioned radius as the `inner radius' and denote it by $\inrad{P}$. Then 
\begin{equation}
\label{eq:inner_radius_of_cell}
\inrad{P}:=\sup\{ r>0\ :\ B(x_c,r)\subset P\text{ for some }x_c\in P\},
\end{equation}
where $B(z,r)$ denotes the open ball with radius $r$ and centre $z$ with respect to the Euclidean norm $\abs{\cdot}_2$. In \eqref{eq:inner_radius_of_cell}, the centre $x_c$ of the corresponding ball is sometimes referred to as the `Chebyshev centre' \cite[Section 4.3.1]{boyd2004convex}. Since the inner radius is invariant under translations of $P$, we may assume without loss of generality that the Chebyshev centre of our polytope is the origin. Let $A\in\mathbb{R}^{m\times d}$ be the matrix whose rows $\{a_i^\top\}_{i=1}^{m}$, $a_i\in\mathbb{R}^d$ are the outer unit normals to the $m$ facets of $P$, and let $b\in\mathbb{R}^{m}$ be the vector of offsets from the origin. Then we may represent $P$ as the set $\{x\in\mathbb{R}^d\ :\ Ax\preccurlyeq b\}$ of solutions to a system of linear inequalities, where $Ax\preccurlyeq b$ means $a_i^\top x\leq b_i$ for $i=1,2,\ldots, m$. Given these assumptions, the inner radius is then defined as the value of the optimisation problem
    \begin{equation}
\label{eq:lin_prog_1}
    \begin{array}{lc}
    \text{maximise} & r \\
    \text {subject to} & Au\preccurlyeq b\\
     & \abs{u}_2 \leq r.
    \end{array}
\end{equation}
Given the Euclidean diameter 
\begin{equation}
\label{eq:diam}
 \diam{P}:=\sup\{\abs{x-y}_2\ :\ x,y\in P\}
\end{equation}
 of a $d$-polytope $P$, we define the \emph{degeneracy ratio} of the polytope $P$ as 
 \begin{equation}
\label{eq:degeneracy_parameter}
\delta(P):=\frac{\inrad{P}}{\diam{P}}.
 \end{equation}
The degeneracy ratio may be defined for any bounded subset $U$ of $\mathbb{R}^d$. Observe that $0\leq\delta(U)\leq 1/2$, where $\delta(U)=0$ if and only if $P$ is of strictly positive codimension and $\delta(U)=1/2$ if and only if $ \inrad{U}=\tfrac{1}{2}\diam{U}$, i.e. if and only if $U$ is a Euclidean ball. Thus the ratio in \eqref{eq:degeneracy_parameter} quantifies how close $P$ is to being `round': if $\delta(P)\approx 0$, then $P$ may be `thin' or `flat'. 

The motivation for defining this quantity is the following. Suppose we are given a $d$-polytope $P$ with $m$ facets $\{F_i\}_{i=1}^{m}$, a vector field $J:P\to\mathbb{R}^d$, and a collection $\{\phi_i\}_{i=1}^{m}$ of scalars given by the surface integrals
\begin{equation*}
\phi_i:=\int_{F_i} J(y)\cdot a_i\mathrm{d}\sigma(y),
\end{equation*}
where $\sigma$ is the $(d-1)$-dimensional surface measure induced by Lebesgue measure, and $F_i$ is the facet for which $a_i\in\mathbb{R}^d$ is the outward unit normal. The problem is to determine whether there exists a vector $\tilde{J}\in\mathbb{R}^d$ such that $\tilde{J}\cdot a_i \sigma(F_i)=\phi_i$, for $i=1,\ldots, m$. That is, we wish to solve 
\begin{quote}
\textbf{Problem}: Given a polytope $P\subset\mathbb{R}^d$ with $m$ facets $\{F_i\}_{i=1}^{m}$ defined by the system $Ax\preccurlyeq b$ of linear inequalities, a vector field $J:P\to\mathbb{R}^d$, and a vector $\phi$ with components given by the surface integrals $\{\phi_i\}_{i=1}^{m}$, determine conditions on $P$ such that there exists a unique solution $J\in\mathbb{R}^d$ to $AJ=\phi$. 
\end{quote}
Given that any $d$-polytope must have at least $d+1$ facets, we observe that $m\geq d+1$, and hence it is not in general true that $\phi$ lies in the column space of $A$. However, by considering the normal equations $A^\top AJ=A^\top \phi$, it follows that there exists a unique solution to $AJ=\phi$ if $A^\top A$ is invertible. Since $A^\top A$ is invertible if and only if its eigenvalues are strictly positive, and since the eigenvalues of $A^\top A$ are the squares of the singular values of $A$, it follows that $A^\top A$ is invertible if and only if the smallest singular value $\sigma_{\min}(A)$ of $A$ is strictly positive. 

The main result of this paper is the following.
\begin{prop}
\label{prop:relates_sigma_and_degeneracy_ratio_vor_cell}
Let $P:=\{x\in\mathbb{R}^d\ :\ Ax \preccurlyeq b\}$ be a bounded convex $d$-polytope. Then $\delta (P) < \sigma_{\min}(A)$.
\end{prop}
Proposition \ref{prop:relates_sigma_and_degeneracy_ratio_vor_cell} characterises a geometric property of a polytope $P$ in terms of an algebraic property of the matrix $A$ that partly determines $P$, by majorising the degeneracy ratio of $P$ in terms of the smallest singular value of $A$. The result can also be used to bound the smallest singular value of the matrix $A$ away from zero, so we may use Proposition \ref{prop:relates_sigma_and_degeneracy_ratio_vor_cell} to solve the aforementioned problem: if the degeneracy ratio $\delta(P)$ of $P$ is strictly positive, then there exists a unique solution $J$ to $AJ=\phi$. Note that Proposition \ref{prop:relates_sigma_and_degeneracy_ratio_vor_cell} cannot be used to bound the smallest singular value of an \emph{arbitrary} rectangular matrix $A$. This is because the vector $b$ of offsets of the polytope is also important, and enters implicitly via the argument $P$ of the degeneracy ratio $\delta(P)$. 

\subsection{Related work}
\label{ssec:related_work}

In the numerical analysis of finite element methods for solving partial differential equations (PDEs) defined over bounded domains, the geometry of the sets with which the bounded domain is discretised plays an important role. One of the main ideas is to impose conditions so that the discretisation sets stay sufficiently regular as the discretisation is refined, in order to ensure convergence with a suitable rate of the finite element approximation of the solution of the PDE. In particular, the inverse of the degeneracy ratio appears in P. Ciarlet's book \cite{ciarlet2002} via the so-called `regularity condition' or `inscribed ball condition': a triangulation of a domain $\Omega\subset\mathbb{R}^2$ is said to yield `regular' finite elements if there exists a constant $\sigma>0$ such that $\diam{K}/\inrad{K}\leq \sigma$ and $\diam{K}\leq 1$ for every triangle $K$ in the triangulation. Apart from the boundedness condition on the diameter, the regularity condition is equivalent to requiring that the degeneracy ratio of each triangle is bounded away from zero. Under this assumption, one can establish an error bound in the finite element approximation of a function \cite[Section 3.1]{ciarlet2002}. A similar condition was introduced in \cite{lin1985FEaccuracy}, while a survey of conditions of this kind is available in \cite{brandts2011angle}. We point out that the inscribed ball condition is known to be sufficient but not necessary for convergence of the finite element method \cite{babuska1976anglecond}, and that other, weaker sufficient conditions such as the `maximum angle condition' \cite{synge1957maxang} are known. 

Similar quantities appear also in the literature on finite volume methods.  For example, in \cite{du2003finiteVolumesVoronoi}, a finite volume method based on Voronoi tessellations of the sphere is constructed, and a `regularity norm' is defined for the Voronoi tessellation by minimising the ratio $\abs{x_i-x_j}/\diam{V_i}$ over all neighbours of each Voronoi cell $V_i$ generated by a point $x_i$ on the sphere, and over all Voronoi cells $\{V_i\}_{i\in I}$ in the tessellation. Note that in both $\abs{\cdot}$ and $\diam{\cdot}$ above are computed with respect to the geodesic metric on the sphere. For any given Voronoi cell $V_i$ with generator $x_i$, the smallest value of $\abs{x_i-x_j}/\diam{V_i}$ over all generators $x_j$ of adjacent Voronoi cells $V_j$ resembles the degeneracy ratio of $V_i$ defined in \eqref{eq:degeneracy_parameter}. This is because if $P=V_i$ is a Voronoi cell with generator $x_i$, and if $\min_j\abs{x_i-x_j}$ is small, then $\inrad{V_i}$ will be small, and vice versa.

To the best of our knowledge, the numerical analysis literature that examines geometric constraints on discretisation sets in terms of degeneracy ratio-like quantities does not consider these quantities from the point of view of singular values of the matrix of outer unit normals. Furthermore, in the literature on convex optimisation, we were likewise unable to find any results that relate the singular values of the matrix of outer unit normals to the `roundness' of the polytope. On the basis of this evidence, it appears that Proposition \ref{prop:relates_sigma_and_degeneracy_ratio_vor_cell} is new.

\section{Proof of the result}

We first fix our notation. Let $A\in\mathbb{R}^{m\times d}$ and $b\in\mathbb{R}^{m}$ denote the matrix with outer unit normals and the vector of offsets of the polytope $P$ respectively. Given any set $A$, $\partial A$ and $\inter{A}$ denote the boundary and interior of $A$ respectively. Given any two distinct points $x$, $y$ in the same Euclidean space, $[x\ y]=\{\lambda x+(1-\lambda)y\ :\ 0\leq \lambda \leq1\}$ denotes the line segment joining $x$ and $y$.

Before proving Proposition \ref{prop:relates_sigma_and_degeneracy_ratio_vor_cell}, we establish the following facts.

\begin{lem}\label{lem:all_bi_positive}
    Let $P\subset\mathbb{R}^d$ be a $d$-dimensional polytope. Then the origin lies in $\inter{P}$ if and only if $0\prec b$.
\end{lem}
\begin{proof}
   Suppose that $0\in \inter{P}$. Then the origin $0\in\mathbb{R}^d$ satisfies the system of linear inequalities, i.e. $0\preccurlyeq b$. To show that $0\prec b$, we need to show that every index $i\in\{1,\ldots,m\}$ satisfies $a_i^\top 0=0<b_i$. Suppose that there exists $i \in \{1,\dots m\}$ such that $b_i = 0$. Then $a_i^{\top}0 = b_i$, which implies that the origin lies on at least one of the supporting hyperplanes of the polytope $P$. Hence $0 \in \partial P$, which contradicts the assumption that $0 \in \inter{P}$. Thus, every $i\in\{1,\ldots,m\}$ satisfies $a_i^\top 0<b_i$. If $0\prec b$, then $b_i>0$ for $i=1,\ldots,m$, so set $\tilde{b}:=\min_i b_i>0$. The ball $B(0,\tfrac{1}{2}\tilde{b})$ is a proper subset of $P$, since for arbitrary $u\in B(0,\tfrac{1}{2}\tilde{b})$, it holds that $a_i^\top u\leq \abs{a_i}_2\abs{u}_2\leq \tfrac{1}{2}\tilde{b}<b_i$ for every $i=1,\ldots,m$. This is equivalent to the system $Au\prec b$ of linear inequalities, which implies that $u\in \inter{P}$.
\end{proof}
\begin{lem}
 \label{lem:equivalent_lin_prog}
 The value of the optimisation problem \eqref{eq:lin_prog_1} is the same as the value of the optimisation problem 
 \begin{equation}
\label{eq:lin_prog_2}
    \begin{array}{lc}
    \text{maximise} & r \\
    \text {subject to} & re\preccurlyeq b
    \end{array}.
\end{equation}
\end{lem}
\begin{proof}
 By the Cauchy-Schwarz inequality and the fact that $\abs{a_i}_2=1$ for all $i=1,\ldots,m$, we have that $\sup\{a_i^{\top}u: \abs{u}_2\leq r \} = r$. Thus, $Au\preccurlyeq re$, where $e\in\mathbb{R}^{m}$ is the constant vector with components all equal to 1, and it follows that $re\preccurlyeq b$.
\end{proof}

\begin{proof}[Proof of Proposition \ref{prop:relates_sigma_and_degeneracy_ratio_vor_cell}]
Since both $\sigma_{\min}(A)$ and $\delta(P)$ are invariant under translations of $P$, we may assume without loss of generality that $0\in\inter{P}$.

	Let $A = U\Sigma V^\top$ denote the singular value decomposition of $A$, and let $\{u_1,\dots,u_m\}$ and $\{v_1,\dots,v_d\}$ denote the column vectors of orthogonal matrices $U\in\mathbb{R}^{m\times m}$ and $V\in\mathbb{R}^{d\times d}$ respectively. The matrix $\Sigma\in\mathbb{R}^{m\times d}$ contains the singular values $\sigma_1\geq \sigma_2\geq \cdots\geq \sigma_d\equiv\sigma_{\min}(A)$ of $A$ on its diagonal and zeros elsewhere. The matrix $A$ can thus be written as 
	\begin{equation*}
	A = \sum_{i=1}^d \sigma_i u_i v_i^\top,
	\end{equation*}
    where we emphasise that $u_d$ and $v_d$ are the left and right singular vectors of $A$ corresponding to the smallest singular value $\sigma_d=\sigma_{\min}(A)$. 
    
    Observe that the intersection of $\operatorname{span}(v_d)$ and the polytope $P$ is a line segment $[f\ g]$, where $f = \lambda_1 v_d\in\partial P$ and $g = \lambda_2 v_d\in\partial P$ for $\lambda_1, \lambda_2 \in  \bb{R}$. Since $0\in \inter{P}$ and $0 \in \operatorname{span}(v_d)$, it follows that $0\in [f\ g]$. Thus $\lambda_1,\lambda_2\neq 0$, and we must have
\begin{equation}\label{eq:signs_of_lambda_1_and_lambda_2_differ}
    \operatorname{sign}(\lambda_1) = -\operatorname{sign}(\lambda_2).
\end{equation}
Observe that the length of the line segment $[f\ g]$ is 
\begin{equation}\label{eq:fg_distance}
    \abs{f-g}_2 = \abs{\lambda_1 v_d - \lambda_2 v_d}_2 = \abs{\lambda_1 - \lambda_2}\abs{v_d}_2 = \abs{\lambda_1 - \lambda_2}.
\end{equation}
Since $f,g \in \partial P$, there exist $i,j \in \{1,\dots, m \}$ such that $(Af)_i =a_i^\top f= b_i$ and $(Ag)_j =a_j^\top g= b_j$. If $i=j$, then $f$ and $g$ lie on the hyperplane $\{x\in\mathbb{R}^d\ :\ a_i^\top x=b_i\}$, which implies that the line segment $[f\ g]$ is contained in this hyperplane; this in turn implies that the origin is contained in this hyperplane, i.e. $0\in\partial P$. Since we assume that $0\in\inter{P}$, it follows by taking the contrapositive that $i$ and $j$ are distinct.

The linear mapping $A$ maps $f$, $g$, and $[f\ g]$ to $Af=\lambda_1\sigma_{\min}(A) u_d$, $Ag=\lambda_2\sigma_{\min}(A)u_d$ and $[Af\ Ag]$ respectively, where the length of the line segment $[Af\ Ag]$ is
\begin{equation}
\label{eq:01}
 \abs{Af-Ag}_2=\sigma_{\min}(A)\abs{f-g}_2=\sigma_{\min}(A)\abs{\lambda_1-\lambda_2}
\end{equation}
by \eqref{eq:fg_distance}. We observed earlier that $(Af)_i=b_i$ and $(Ag)_j=b_j$ for distinct $i,j\in\{1,\ldots,m\}$. In addition, we note that
\begin{equation*}
    (Ag)_i = (\sigma_d \lambda_2 u_d)_i = \frac{\lambda_2}{\lambda_1}(\sigma_d \lambda_1 u_d)_i  = \frac{\lambda_2}{\lambda_1}(Af)_i  = \frac{\lambda_2}{\lambda_1}b_i.
\end{equation*}
Let $b' \in \bb{R}^m$ be such that $b'_k := (Ag)_k$ for $k\in \{1,\dots m\}\setminus \{i\}$ and $b'_i = (Af)_i=b_i$. Consider the triangle in $\mathbb{R}^{m}$ with vertices $Af$, $Ag$ and $b'$. This triangle is right-angled with right angle at the vertex $b'$, because the line segment between $Af$ and $b'$ is contained in the hyperplane $\{x\in \bb{R}^m\ :\ x_i = b_i\}$, the line segment between $Ag$ and $b'$ is contained in the hyperplane $\{x\in \bb{R}^m\ :\ x_j = b_j\}$, and these two hyperplanes are orthogonal to each other because $i$ and $j$ are distinct. Since $Ag$ and $b'$ differ only in the $i$-th coordinate, the length of the line segment $[Ag\ b']$ is 
\begin{equation*}
\abs{Ag-b'}_2= \abs{ b_i\frac{\lambda_2}{\lambda_1}-b_i} = b_i\left(1+\abs{\frac{\lambda_2}{\lambda_1}}\right)>b_i,
\end{equation*}
where we used Lemma \ref{lem:all_bi_positive}, \eqref{eq:signs_of_lambda_1_and_lambda_2_differ}, and the fact that $\lambda_1,\lambda_2$ are nonzero and of opposite sign. The hypotenuse of the right-angled triangle is given by the line segment $[Af\ Ag]$. Comparing the length of the hypotenuse $[Af\ Ag]$ to that of $[Ag\ b']$ and using the preceding inequality, we have 
\begin{equation}\label{eq:sigma_d_b_relation}
	\sigma_d\abs{\lambda_1 - \lambda_2}=\abs{Af-Ag}_2>\abs{Ag-b'}_2= b_i\left(1 + \abs{\frac{\lambda_2}{\lambda_1}}\right)>b_i,
\end{equation}
where we used \eqref{eq:01}. Now it follows that 
\begin{equation}
 \label{eq:02}
 \sigma_{\min}(A)\diam{P}\geq \sigma_{\min}(A)\abs{f-g}_2=\abs{Af-Ag}_2>b_i,
\end{equation}
where we used the definition of the diameter, \eqref{eq:01}, and \eqref{eq:sigma_d_b_relation}.
Recall from Lemma \ref{lem:equivalent_lin_prog} that $\inrad{P}$ is the value of the optimisation problem \eqref{eq:lin_prog_2}. The constraint in \eqref{eq:lin_prog_2} implies that $\inrad{P} \leq b_i$ for every $i\in\{1,\ldots,m\}$. By \eqref{eq:02},
\begin{equation*}
 \sigma_{\min}(A)\diam{P}>b_i\geq \inrad{P},
\end{equation*}
and applying the definition \eqref{eq:degeneracy_parameter} of $\delta(P)$ completes the proof.
\end{proof}

\bibliographystyle{plain} 
\bibliography{degeneracy_ratio}

\end{document}